\documentclass[a4paper,12pt]{article}

\usepackage[centertags]{amsmath}
\usepackage{amsfonts}
\usepackage{amssymb}
\usepackage{amsthm}
\usepackage{dsfont}
\usepackage{tikz, subfigure}
\usepackage{verbatim}
\usepackage{graphicx}

\newtheorem{condition}{Condition}[section]

\newtheorem{theorem}{Theorem}[section]
\newtheorem{lemma}{Lemma}[section]
\newtheorem{cor}{Corollary}[section]
\newtheorem{prop}{Proposition}[section]
\newtheorem{definition}{Definition}[section]

\newcommand{\gl}{{\lambda}}
\newcommand{\gk}{{\kappa}}

\newcommand{\gb}{{\beta}}

\DeclareMathOperator{\free}{free}
\DeclareMathOperator{\succf}{succ}

\begin{document}
\title{Towards Absolutely Continuous Bernoulli Convolutions}
\author{Alex Batsis and Tom Kempton}
\maketitle

\begin{abstract}{\noindent 
We show how to turn the question of the absolute continuity of Bernoulli convolutions into one of counting the growth of the number of overlaps in the system. When the contraction parameter is a hyperbolic algebraic integer, we turn this question of absolute continuity into a question involving the ergodic theory of cocycles over domain exchange transformations.}
\end{abstract}

\section{Introduction}
Bernoulli convolutions are a simple family of overlapping self-similar measures. For $\beta\in(1,2)$ the Bernoulli convolution $\nu_\beta$ is defined be the weak$^*$ limit of the sequence $\nu_{\beta,n}$ of probability measures given by 
\[
\nu_{\beta,n}=\sum_{a_1\cdots a_n\in\{0,1\}^n} \frac{1}{2^n}\delta_{\sum_{i=1}^na_i\beta^{-i}}.
\]

The question of the absolute continuity of Bernoulli convolutions goes back to work of Erd\H{o}s in 1939 \cite{ErdosPisot}, in which it was shown that the $\nu_{\beta}$ is singular when $\beta$ is a Pisot number. These remain the only known examples of singular Bernoulli convolutions. In the other direction, Garsia, Varj\'u and Kittle have each given examples of classes of absolutely continuous Bernoulli convolutions associated with algebraic parameters \cite{GarsiaAC, VarjuAC, Kittle}. Solomyak showed that the set of $\beta\in(1,2)$ giving rise to singular Bernoulli convolutions has Lebesgue measure zero \cite{SolomyakAC}, this result was improved by Shmerkin who showed that the set has Hausdorff dimension zero \cite{ShmerkinAC}. 

If instead of asking for absolute continuity of $\nu_{\beta}$ we ask whether $\dim_H(\nu_{\beta})=1$ then a lot more is known, mainly stemming from work of Hochman \cite{HochmanInverse}. Several recent articles give conditions under which the Bernoulli convolution associated to an algebraic $\beta$ has dimension one \cite{VarjuBreuillard1,VarjuBreuillard2, HKPS} or show that the Hausdorff dimension can be computed \cite{AFKP}. Most significantly, Varj\'u has shown that $\dim_H(\nu_{\beta})=1$ whenever $\beta$ is transcendental \cite{VarjuTranscendental}. Finally we mention recent papers of Feng and Feng and of Kleptsyn, Pollicott and Vytnova which give remarkable lower bounds for the $\dim_H(\nu_{\beta})$ which hold for all $\beta\in(1,2)$ \cite{FengFeng, KleptsynPollicottVytnova}. For a recent summary see \cite{VarjuSummary2}. 

In this article we give new conditions for the absolute continuity of Bernoulli convolutions. Our first result frames the question of absolute continuity in terms of counting overlaps. Let $\mathcal N_n$ be the number of overlaps at the $n$th level of the construction of the Bernoulli convolution. This is equal to the number of pairs of words $a_1\cdots a_n, b_1,\cdots b_n\in\{0,1\}^n$ for which $\left|\sum_{i=1}^n a_i\beta^{n-i}-\sum_{i=1}^n b_i\beta^{n-i}\right|<\frac{1}{\beta-1}$.

\begin{prop}\label{introprop}[Special Case of Theorem \ref{CountingTheorem}]
If there exists $C>0$ such that $\mathcal N_n\leq C\left(\frac{4}{\beta}\right)^n$ for all $n\in\mathbb N$ then the Bernoulli convolution $\nu_{\beta}$ is absolutely continuous.
\end{prop}

The remainder of the article focuses on hyperbolic algebraic integers and is spent recasting this counting question in terms of equidistribution of a family of probability measures with respect to Lebesgue measure. This family of measures is supported on a cut and project set, which allows us to turn the question of the absolute continuity of certain Bernoulli convolutions into a question relating to the ergodic theory of cocycles over domain exchange transformations. Our hope is that, with further work, our techniques will give rise to a proof that the Bernoulli convolution $\nu_{\beta}$ is absolutely continuous whenever $\beta\in(1,2)$ is algebraic and has at least one Galois conjugate larger than one in absolute value, with no Galois conjugates having absolute value one. Our final theorem is the following.

\begin{theorem}\nonumber[Stated Precisely as Theorem 
\ref{FinThm}.]
Assume that $\beta\in(1,2)$ is an algebraic integer with no Galois conjugates of absolute value one and at least one real Galois conjugate of absolute value larger than one. There are checkable assumptions (see Theorem \ref{FinThm}) under which there exist a fractal $\mathcal R$, an interval $I$, a  domain exchange transformation $T:I\times\mathcal R\to I\times \mathcal R$ and a function $f:\mathcal R\to \mathbb R^+$ such that, if the projection onto $I$ of the sequence of measures
\[
\sum_{i=1}^n f(0)f(T(0))\cdots f(T^{n-1}(0))\delta_{T^{n-1}(0)}
\]
converges to Lebesgue measure sufficiently quickly then the Bernoulli convolution $\nu_{\beta}$ is absolutely continuous.
\end{theorem}
If the function $f$ took values in a compact group $K$ then the Santos-Walkden version of the Wiener-Wintner ergodic theorem \cite{SantosWalkden} would give us the convergence that we need. As it is, further work on the ergodic theory of cocycles over domain exchange transformations is needed to use our techniques to prove that certain Bernoulli convolutions are absolutely continuous.

We illustrate our results by first looking at a particular example.

\subsection{A First Example:}\label{example}
Let $\beta \approx 1.513$ satisfy $\beta^4=\beta^3+\beta^2-\beta+1$. Then $\beta$ has one real Galois conjugate $\beta_2\approx -1.179$ and a pair of complex Galois conjugates which are less than one in modulus. We chose this example because it has no Galois conjugates of absolute value one (essential for our techniques) and because it is of small degree with only one Galois conjugate larger than one in modulus (which makes things easier to compute and to visualise).

Our first result, a special case of Theorem \ref{CountingTheorem}, gives conditions for the absolute continuity of $\nu_{\beta}$ in terms of the growth of the total number of overlaps at the $n$th level of the construction of the Bernoulli convolution. This is stated as Proposition \ref{introprop} above.



Unfortunately, estimating $\mathcal N_n$ is difficult. The bulk of this paper is dedicated to giving upper bounds via a geometric construction. 

We define the measure $\mu_n$ on $I:=\left[\frac{-1}{\beta-1},\frac{1}{\beta-1}\right]$ by
\[
\mu_n(A)=\#\{a_1\cdots a_n, b_1\cdots b_n\in\{0,1\}^n : \sum_{i=1}^n (a_i-b_i)\beta^{n-i}\in A\}.
\]
Then $\mathcal N_n=\mu_n(I)$. 

We want to understand the ratio $\frac{\mathcal N_{n+1}}{\mathcal N_n}$. Given $a_1\cdots a_n, b_1\cdots b_n$ contributing to the count for $\mathcal N_n$, we ask how many of the four choices of $a_{n+1}, b_{n+1}\in\{0,1\}^2$ give rise to a pair $a_1,\cdots a_{n+1}, b_1\cdots b_{n+1}$ contributing to the count for $\mathcal N_{n+1}$. This boils down to the number of $a_{n+1}, b_{n+1}$ for which
\[
\beta\left(\sum_{i=1}^n (a_i-b_i)\beta^{n-i}\right)+(a_{n+1}-b_{n+1})\in I,
\]
which in turn depends only on the value of $\sum_{i=1}^n (a_i-b_i)\beta^{n-i}$. Using this, we show in Section \ref{Sec3} that the ratio $\frac{\mathcal N_{n+1}}{\mathcal N_n}$ can be expressed as the integral of a step function $g$ with respect to the measure $\mu_n$. This yields the following corollary.
\begin{prop}\label{P1.2}
Suppose that the measures $\mu_n$ equidistribute with respect to Lebesgue measure on $I$ with certain rate (made precise in Theorem \ref{weakconvwithrate} and the comments afterwards. Then the Bernoulli convolution $\nu_{\beta}$ is absolutely continuous.
\end{prop}



If one draws the points contributing to the count for $\mathcal N_n$, that is if one draws the set \[\left\{\sum_{i=1}^n (a_i-b_i)\beta^{n-i}:\mbox{ each }a_i, b_i\in\{0,1\}\right\}\cap I\] then no structure is apparent, although the set of points becomes increasingly dense as $n$ increases. Similarly, the measures $\mu_n$ do not seem to have any discernable structure when viewed in one dimension. 

If however, one includes a second coordinate using the other Galois conjugate larger than one in modulus, then one uncovers the highly structured set
\[
X_n=\left\{\sum_{i=1}^n \left((a_i-b_i)\beta^{n-i}, (a_i-b_i)\beta_2^{n-i}\right):\mbox{ each }a_i, b_i\in\{0,1\}\right\}\cap\left( I \times \mathbb R\right)
\]
We have plotted this set below for $n=6$.

\begin{figure}[htbp]
\centerline{\includegraphics[scale=.2]{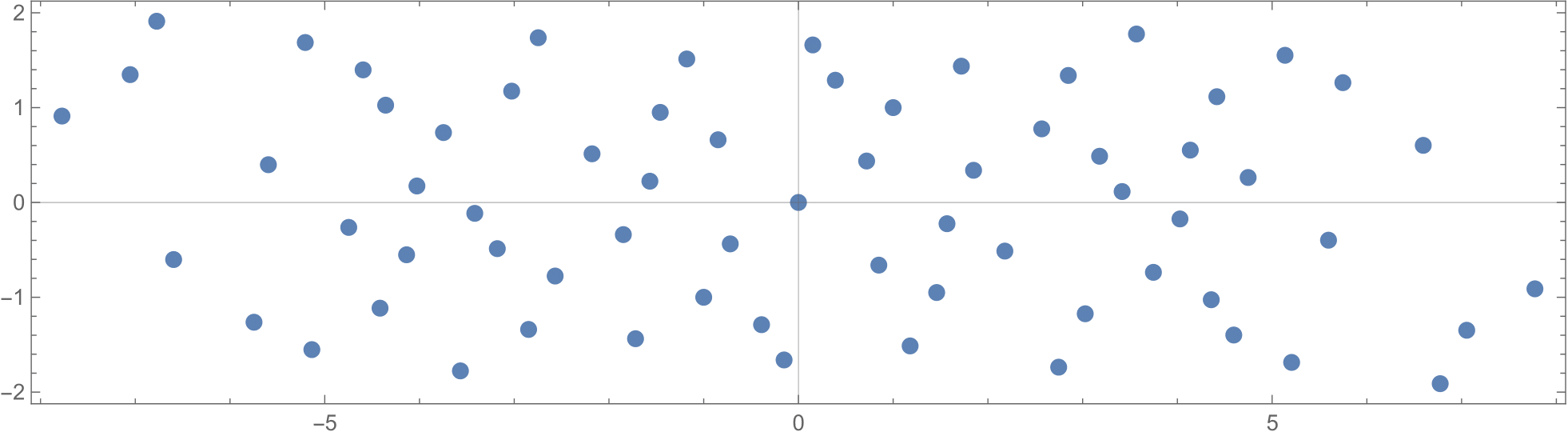}}
\caption{The set $X_6$ reflected across the diagonal.}
\label{fig}
\end{figure}

The measure $\mu_n$ lifts naturally to a measure on $X_n$. As $n$ grows, $X_n$ expands to fill the set
\[
X=\left\{\sum_{i=1}^n \left((a_i-b_i)\beta^{n-i}, (a_i-b_i)\beta_2^{n-i}\right):n\in\mathbb N, \mbox{ each }a_i, b_i\in\{0,1\}\right\}\cap\left(I\times\mathbb R\right)
\]
which is uniformly discrete and relatively dense in the strip $\left(I\times\mathbb R\right)$. In fact $X$ is a cut and project set where the cut and project scheme uses a window involving the Galois conjugates less than one in modulus, it can be constructed by a method similar to that of the Rauzy fractal \cite{ArnouxHarriss}. 

In order to estimate $\mathcal N_n$ we are left with two problems, firstly to work out which elements of $X$ are in $X_n$, and secondly to work out $\mu_n(x)$ for points $(x,y)\in X_n$. The first problem is easy, we use the $y$-coordinate 
$\sum_{i=1}^n (a_i-b_i)\beta_2^{n-i}$ as a proxy for the smallest $n$ for which $(x,y)\in X_n$, it is certainly true that
\[
X_n\subset \{(x,y)\in X : |y|\leq \sum_{i=1}^n \beta_2^{n-i}\}
\]
and this estimate is good enough for us. 

The second problem is much harder, and we rely heavily on our work \cite{Measures}. We use that there exists $\alpha>1$ such that, for each $(x,y)\in X$, $\mu(x):=\lim_{n\to\infty}\frac{1}{\alpha^n}\mu_n(x)$ exists. The key result of section \ref{Sec4} gives the following corollary, stated more precisely in Theorem \ref{moving to w}. 

\begin{prop}\label{P1.3}[Special Case of Theorem \ref{moving to w}]
Suppose that the sequence of measures
\[
\sum_{(x,y)\in X:y \in [-\sum_{i=1}^n \beta_2^{n-i},\sum_{i=1}^n \beta_2^{n-i}]} \mu(x) \delta_{x},
\]
once renormalised to have mass one, converges with certain rate to Lebesgue measure. Then the Bernoulli convolution $\nu_{\beta}$ is absolutely continuous. 
\end{prop}

The convergence to the Lebesgue measure of sequence of measures above is consistent with numerical evidence. Table \ref{tab:table1} shows the Wasserstein distance of 

\[
\sum_{(x,y)\in X:y \in [-n/(\beta-1),n/(\beta-1)]} \mu(x) \delta_{x},
\]

once normalised to have mass one, to the Lebesgue measure for $n=1,...,20$.

\begin{table}[ht]
  \begin{center}\small

    \begin{tabular}{r|c} 
      \text{n} & \text{$W_1(\cdot,\operatorname{Leb})$} \\
     
      \hline
   
1&0.0257383\\
2&0.0154008\\
3&0.0079060\\
4&0.0068856\\
5&0.0065858\\
6&0.0048812\\
7&0.0038639\\
8&0.0053756\\
9&0.0047376\\
10&0.0049352\\
11&0.0040242\\
12&0.0054624\\
13&0.0030473\\
14&0.0033527\\
15&0.0021562\\
16&0.0028536\\
17&0.0021284\\
18&0.0031695\\
19&0.0018788\\
20&0.0016524\\

    \end{tabular}\caption{Evidence for an equidistribution property of the measures $\mu_n$. }\label{tab:table1}
  \end{center}
\end{table}

One can study the support of the sequence of measures defined in Proposition \ref{P1.3} using domain exchange transformations, in much the same way that one studies greedy $\beta$ expansions using the Rauzy fractal. We also proved in \cite{Measures} that one can study the measures (rather than just the support) using a cocycle over this domain exchange transformation. This yields Theorem \ref{FinThm} which gives a condition for the absolute continuity of the Bernoulli convolution in terms of the ergodic theory of cocycles over domain exchange transformations.

\section{A First Condition for Absolute Continuity}\label{Sec2}
There has been a lot of progress in recent years in showing that certain Bernoulli convolutions have dimension one. For algebraic parameters this has based on understanding Garsia entropy, which counts the number of exact overlaps in the level $n$ approximations to the Bernoulli convolution. In this section we explain how good estimates in the total number of overlaps (including partial overlaps) in the level $n$ approximation to the Bernoulli convolution would allow one to understand absolute continuity.

Our starting point is the article \cite{CountingBeta} of the second author, in which two simple observations were made. The first is that if a self-similar measure $\nu$ is absolutely continuous, then the similarity equation which $\nu$ satisfies gives rise to a similarity equation for its density $h$. Furthermore, the measure $\nu$ is absolutely continuous if and only if there exists an $L^1$ function satisfying this density self-similarity equation. In the case of Bernoulli convolutions associated to a parameter $\beta\in(1,2)$ the statement becomes that the Bernoulli convolution is absolutely continuous if and only if there exists a non-negative $L^1$ function $h:\mathbb R\to\mathbb R$ such that
\[
h(x)=\frac{\beta}{2}(h(\beta x)+h(\beta x-1)).
\]
The second observation of \cite{CountingBeta} was that one can study the existence of solutions to such equations in terms of functions which count the number of codings of each point $x$ in the level n-construction of the self-similar measure. 

In this section we generalise both of these ideas to measures on self-affine carpets with contraction rates in different directions corresponding to Galois conjugates of $\beta$, these measures are higher dimensional generalisations of Bernoulli convolutions. We also convert the second observation described above into one involving counting the total number of overlaps in the self-affine construction. When the self-affine measures we study are projected onto their first coordinate they give rise to the Bernoulli convolution, and so absolute continuity of these self-affine measure implies the absolute continuity of the Bernoulli convolution.

\subsection{The Self-Affine Case}

Let $\beta\in(1,2)$ be a hyperbolic algebraic integer. 

We will be interested in diagonal self-affine sets with contraction parameters associated with all but one of the Galois conjugates of $\beta$ of absolute value larger than one. For this reason we number the Galois conjugates of $\beta$ in an unusual way, let $\beta$ have Galois conjugates $\beta=\beta_1,...,\beta_d,\beta_{d+1},\cdots,
\beta_{d+s},\beta_{d+s+1}$ where $|\beta_1|,...,|\beta_d|>1$, $|\beta_{d+1}|,...,|\beta_{d+s}|<1$ and $\beta_{d+s+1}\in\mathbb{R}\setminus[-1,1]$. 

In this section we will focus on $\beta_1,...,\beta_d$. For $z\in\mathbb C$ set $\mathbb F_z=\mathbb R$ when $z\in\mathbb R$ and $\mathbb F_z=\mathbb C$ when $z\in\mathbb C\setminus\mathbb R$. Further define
\begin{align*}
\mathbb{K}:=\prod_{i=1}^{d}\mathbb{F}_{\gb_i},
\end{align*}
For $i\in\mathbb N$ we define $T_i:\mathbb K\to\mathbb K$ by
\[T_i(x_1,...,x_{d})=(\beta_1x_1+i,...,\beta_{d}x_{d}+i).\]

For $j\in\{1,\cdots d\}$ let
\begin{align*}
I_{\beta_j}^+=
\begin{cases}
\left[0,\frac{1}{\beta_j-1}\right],\quad \beta_j\in(1,\infty)\\
\left\{x\in\mathbb{R}:|x|\in\left[0,\frac{1}{|\beta_j|-1}\right]\right\},\quad x\in(-\infty,-1)\\
\left\{z\in\mathbb{C}:|z|\in\left[0,\frac{1}{|\beta_j|-1}\right]\right\},\quad z\in\mathbb{C}\setminus\mathbb{R}
\end{cases}
\end{align*}
and 

\[
I^+=I^+_{\beta_1}\times \cdots \times I^+_{\beta_d}.
\]
Define the self-affine measure $\nu_{\underline{\beta}}$ on $\mathbb{K}$ by
\begin{equation}\label{SelfAffine}
\nu_{\underline{\beta}}=\frac{1}{2}\left(\nu_{\underline{\beta}}\circ T_0+\nu_{\underline{\beta}}\circ T_{-1}\right).
\end{equation}

Note that the maps $T_i$ are expanding, and $\nu_{\underline{\beta}}$ is the measure associated to contractions $T_0^{-1}, T_{-1}^{-1}$. This measure has support contained in $I^+$. If $\nu_{\underline{\beta}}$ is absolutely continuous then  $\nu_{\beta}$ is absolutely continuous, we aim to prove the absolute continuity of $\nu_{\underline{\beta}}$.




Define an operator $P$ on functions $f:\mathbb{K}\to\mathbb R$ by letting 
\[
Pf=\frac{|\beta_1\cdot...\cdot\beta_{d}|}{2}(f\circ T_0+f\circ T_{-1}).
\]

P preserves the space of non-negative functions that vanish outside $I^+$ and have integral one. $P$ is a linear operator, and in particular if $f$ is a fixed point of $P$ then $cf$ is also a fixed point of $P$ for any constant $c>0$, thus if P has a fixed point of positive finite integral then it has a fixed point of integral one.

\begin{prop}\label{densityprop}
Suppose that $P$ has a fixed point which has positive finite integral. Then the self-affine measure $\nu_{\underline{\beta}}$ is absolutely continuous and the fixed point of $P$ of integral one is the density of $\nu_{\underline{\beta}}$.
\end{prop}

\begin{proof}
By integrating the fixed point $f$ of $P$ with integral one, we get a probability measure $\nu'$ on $I^+.$ In order to check that $\nu'=\nu_{\underline{\beta}}$ we need only check that $\nu'$ satisfies the self-affinity equation \ref{SelfAffine}, and so it is enough to check that for any $A\subset I^+$ we have
\[
\nu'(A)=\frac{1}{2}\left(\nu'(T_0(A))+\nu'(T_{-1}(A))\right).
\]
This then follows immediately from the equation $Pf=f$ using that
\begin{eqnarray*}
\nu'(A)&=&\int_A f(x_1,\cdots, x_{d})d(x_1,\cdots x_{d})\\
&=&\int_A Pf(x_1,\cdots, x_{d})d(x_1,\cdots x_{d})\\
&=& \frac{|\beta_1\cdot...\cdot\beta_{d}|}{2} \int_A f(T_0(x_1,\cdots, x_{d}))+f(T_{-1}(x_1,\cdots, x_{d})) d(x_1,\cdots, x_{d})\\
&=&\frac{1}{2}\left(\int_{T_0(A)}f(x_1,\cdots, x_{d})d(x_1,\cdots x_{d})+\int_{T_{-1}(A)}f(x_1,\cdots, x_{d})d(x_1,\cdots x_{d})\right)\\
&=&\frac{1}{2}\left(\nu'(T_0(A))+\nu'(T_{-1}(A))\right).
\end{eqnarray*}
\end{proof}

Our goal now is to construct $L^1$ functions which satisfy $Pf=f$. Let functions $f_n$ be given by
\[
f_n:=P^n(\chi_{I^+})
\]
Here $f_n(x_1,\cdots,x_{d})$ gives the number of words $a_1,\cdots,a_n\in\{0,-1\}^n$ for which $T_{a_n}\circ \cdots \circ T_{a_1}(x_1,\cdots, x_d)$ remains in the region $I^+$, multiplied by $\left(\frac{|\beta_1\cdot...\cdot\beta_{d}|}{2}\right)^n$. Equivalently, if we consider the iterated function system on $I^+$ with contractions $T_0^{-1}$, $T_1^{-1}$ then $f_n(x_1,\cdots x_{d})$ counts the number of words $a_1\cdots a_n$ for which $T_{a_1}^{-1}\circ \cdots \circ T_{a_n}^{-1}(I^+)$ covers $(x_1\cdots,x_{d})$, again multiplied by $\left(\frac{|\beta_1\cdot...\cdot\beta_{d}|}{2}\right)^n$.

Since the operator $P$ preserves integral, each $f_n$ has integral equal to the integral of $f_0$, which is the area of $I^+$.

\begin{lemma}\label{convlemma}
Suppose that there exists a uniform constant $C$ such that $||f_n||_2:=\int_{I^+} (f_n(x_1,\cdots x_{d}))^2 d(x_1,\cdots,x_{d})<C$ for all $n\in\mathbb N$. Then $P$ has a fixed point $h$ of integral one and with bounded $L^2$ norm.
\end{lemma}

\begin{proof}
Define
\[
g_n(x_1,\cdots,x_{d}):=\frac{1}{n}\sum_{k=1}^n f_k(x_1,\cdots,x_{d}).
\]
then each $g_n$ also has $||g_n||_2<C$ so, since balls are weakly compact in Hilbert spaces, there is a subsequence of $g_n$ that converges weakly to some $g\in L^2(I^+)$ with $||g||_2\leq C$. Hence  by the Banach-Saks theorem there is a subsequence $g_{n_\gk}$ of $g_n$ such that

\begin{align*}
\left|\left|g-\frac{1}{n}\sum_{\gk=1}^ng_{n_\gk}\right|\right|_2\rightarrow0.
\end{align*}

Furthermore 
\[
||g_\gk-P(g_\gk)||_2=\frac{1}{\gk}||f_1-f_{\gk+1}||_2<\frac{2C}{\gk}
\]

so

\begin{align*}
    \left|\left|\frac{1}{n}\sum_{\gk=1}^ng_{n_\gk}-P\left(\frac{1}{n}\sum_{\gk=1}^ng_{n_\gk}\right)\right|\right|_2&=
    \left|\left|\frac{1}{n}\sum_{\gk=1}^ng_{n_\gk}-\frac{1}{n}\sum_{\gk=1}^nP(g_{n_\gk})\right|\right|_2\\&
    \leq\frac{1}{n}\sum_{\gk=1}^n||g_{n_\gk}-P(g_{n_\gk})||_2\\&
    \leq\frac{1}{n}\sum_{\gk=1}^n\frac{2C}{n_\gk}.
\end{align*}

Letting $n$ go to infinity in the inequality above we get $||g-P(g)||_2=0$ and so $g$ is a fixed point of $P$. Finally, since $g$ is the limit of a sequence of functions of fixed positive finite integral and $||g||_2\leq C$ we conclude that $g$ has positive finite integral, and so we can normalise it to give a function $h$ of integral $1$.

\end{proof}

We now explain how to bound $||f_n||_2$ in terms of the total number of overlaps at level $n$ of the iterated function system $\{T_0^{-1}, T_{-1}^{-1}\}$. Let
\[
\mathcal N_n:=\#\left\{a_1\cdots a_n, b_1\cdots b_n\in\{0,-1\}^{2n}:T_{a_1}^{-1}\circ\cdots T_{a_n}^{-1}(I^+)\cap T_{b_1}^{-1}\circ\cdots T_{b_n}^{-1}(I^+)\neq\varnothing\right\}.  
\]
The question of whether these contracted regions overlap for given $a_1,\cdots,a_n, b_1,\cdots,b_n$ can be phrased in terms of the forward image of the origin $\underline 0$. 

This gives
\begin{eqnarray*}
\mathcal N_n&=& \#\large\{a_1\cdots a_n, b_1\cdots b_n\in\{0,-1\}^{2n} :\left| T_{a_1}\circ \cdots \circ T_{a_n}(\underline 0)-T_{b_1}\circ \cdots \circ T_{b_n}(\underline 0)\right| \\
&\in & I_{\beta_1}\times...\times I_{\beta_d}\large\}\\
&=& \#\large\{a_1\cdots a_n, b_1\cdots b_n\in\{0,1\}^{2n} : \left|\sum_{i=1}^n (a_i-b_i)\beta_j^{n-i}\right|\in I_{\beta_j} \text{ for each } \\
& &j\in\{1,\cdots,d\}\large\}.
\end{eqnarray*}

where 

\begin{align*}
I_{\beta_j}=
\begin{cases}
\left[\frac{-1}{\beta_j-1},\frac{1}{\beta_j-1}\right],\quad \beta_j\in(1,\infty)\\
\left\{x\in\mathbb{R}:|x|\in\left[0,\frac{2}{|\beta_j|-1}\right]\right\},\quad x\in(-\infty,-1)\\
\left\{z\in\mathbb{C}:|z|\in\left[0,\frac{2}{|\beta_j|-1}\right]\right\},\quad z\in\mathbb{C}\setminus\mathbb{R}
\end{cases}
\end{align*}

for $\{1,\cdots,d\}$.

\begin{prop}\label{estimateprop}
We have \[||f_n||_2\leq \lambda(I^+) \left(\frac{|\gb_1\cdot...\cdot\gb_{d}|}{4}\right)^n\mathcal N_n\]
\end{prop}

\begin{proof}
Notice that 
\begin{align*}
P^nf=\left(\frac{|\gb_1\cdot...\cdot\gb_{d}|}{2}\right)^n\sum_{a_1,...,a_n\in\{0,-1\}}f\circ T_{a_1}\circ...\circ T_{a_n}
\end{align*}
So we have 
\begin{align*}
||f_n||_2&=\int_{I^+}f_n(x)f_n(x)dx\\
&= \int_{I^+}\left(     \left(\frac{|\gb_1\cdot...\cdot\gb_{d}|}{2}\right)^n\sum_{a_1,...,a_n\in\{0,-1\}}\chi_{I^+}\circ T_{a_1}\circ...\circ T_{a_n}  \right)\\&
\left(\left(\frac{|\gb_1\cdot...\cdot\gb_{d}|}{2}\right)^n\sum_{b_1,...,b_n\in\{0,-1\}}\chi_{I^+}\circ T_{a_1}\circ...\circ T_{a_n}\right)dx\\
&=\int_{I^+}\left(\frac{|\gb_1\cdot...\cdot\gb_{d}|^2}{4}\right)^n\sum_{a_1,...a_n,b_1,...,b_n}\chi_{I^+}\circ T_{a_1}\circ...\circ T_{a_n}\cdot \chi_{I^+}\circ T_{b_1}\circ...\circ T_{b_n}dx\\
&=\left(\frac{|\gb_1\cdot...\cdot\gb_{d}|^2}{4}\right)^n\sum_{a_1,...a_n,b_1,...,b_n}\int_{I^+}\chi_{I^+}\circ T_{a_1}\circ...\circ T_{a_n}\cdot \chi_{I^+}\circ T_{b_1}\circ...\circ T_{b_n}dx
\end{align*}
Notice that in the bound for  $||f_n||_2$ given above we need to keep only the terms for $a_1,...,a_n,b_1,...,b_n$ such that $\chi_{I^+}\circ T_{a_1}\circ...\circ T_{a_n}\cdot \chi_{I^+}\circ T_{b_1}\circ...\circ T_{b_n}\neq0$, i.e. those $a_1,\cdots a_n,b_1\cdots,b_n$ involved in the definition of $\mathcal N_n$. Furthermore, by noticing that $\int_{I^+}\chi_{I^+}\circ T_{a_1}\circ...\circ T_{a_n}\cdot \chi_{I^+}\circ T_{b_1}\circ...\circ T_{b_n}dx$ is at most $\gl(I^+)|\gb_1\cdots\gb_{d}|^{-n}$, we end up with

\begin{align*}
||f_n||_2\leq \gl(I^+) \left(\frac{|\gb_1\cdot...\cdot\gb_{d}|}{4}\right)^n\mathcal N_n
\end{align*}
as required. 
\end{proof}

Combining Proposition \ref{densityprop}, Lemma \ref{convlemma} and Proposition \ref{estimateprop} gives the following theorem.

\begin{theorem}\label{CountingTheorem}
Suppose that the total number $\mathcal N_n$ of overlaps in the $n$th level of the iterated function system $T_0, T_1$ satisfies that 
\[
\mathcal N_n\leq C\left(\frac{4}{|\gb_1\cdot...\cdot\gb_{d}|}\right)^n
\]
for some constant $C>0$ and for each $n\in\mathbb N$. Then the corresponding self-affine measure $\nu_{\underline{\beta}}$ is absolutely continuous.
\end{theorem}

We have stated Theorem \ref{CountingTheorem} for a measure rectangular self-affine set with contraction rates associated to $\beta_1,\cdots \beta_d$ which were all Galois conjugates, since this is how we will apply the result in later sections, but it is worth noting that assumptions on the contraction rates were not used in this section and the theorem holds for any set of contraction rates $\beta_1,\cdots,\beta_d$. 

\section{Measures on the distance set}\label{Sec3}
In this section we turn from counting the total number of overlaps at level $n$ of our construction to studying the distribution of the overlaps. We define a measure $\mu_n$ describing how the overlaps are distributed, and show that if $\mu_n$ converges sufficiently quickly to Lebesgue measure then the conditions of Theorem \ref{CountingTheorem} are met and so $\nu_{\beta}$ is absolutely continuous. In fact we show that if the integral of a particular step function $g$ with respect to $\mu_n$ converges to the integral of $g$ with respect to Lebesgue measure sufficiently quickly then $\nu_{\beta}$ is absolutely continuous, see Theorem \ref{weakconvwithrate} and the comments afterwards.

Theorem \ref{CountingTheorem} involves counting all pairs $a_1,\cdots, a_n, b_1, \cdots, b_n\in\{0,1\}^{2n}$ for which

\[\left|\sum_{i=1}^n (a_i-b_i)\beta_j^{n-i}\right|\in I_{\beta_j} \mbox{ for each }j\in\{1,\cdots,d\}\]

If we let $\underline{\beta}:=(\beta_1,\cdots,\beta_{d})$, $\underline{\beta^n}:=(\beta_1^n, \cdots, \beta_{d}^n)$, and 
\[
I=I_{\beta_1}\times...\times I_{\beta_d}
\]
we are counting the number of pairs $a_1,\cdots, a_n, b_1,\cdots, b_n$ for which
\[
\sum_{i=1}^n (a_i-b_i)\underline \beta^{n-i}\in I.
\]
Let $\mathcal D_n\subset\{0,1\}^{2n}$ be the set of such pairs $a_1,\cdots, a_n, b_1, \cdots, b_n$. It is useful for us to put a measure on the set of such differences. Let
\[
\mu_n:=\sum_{\left\{a_1\cdots a_n, b_1\cdots b_n \in \mathcal D_n\right\}} \delta_{\sum_{i=1}^n (a_i-b_i)\underline{\beta^{n-i}}},
\]
for $n\geq 1$. This is a sum of weighted Dirac masses, supported on the set $I$,  with total mass $\mathcal N_n$. 

In going from $\mathcal N_n$ to $\mathcal N_{n+1}$ it is useful to note that
\[
\sum_{i=1}^{n+1} (a_i-b_i)\beta_j^{(n+1)-i}=\beta_j\left(\sum_{i=1}^n (a_i-b_i)\beta_j^{n-i}\right)+(a_{n+1}-b_{n+1}),
\] 
with the difference $(a_{n+1}-b_{n+1})$ taking value $1,-1,$ or $0$. There are two different ways of getting value $0$ here, we can have $a_{n+1}=b_{n+1}=0$ or $a_{n+1}=b_{n+1}=1$.

Define an operator $\Phi$ on the space of measures on $I$ by letting
\[
(\Phi(\mu))(A):=\mu\left(T_1^{-1}(A)\right)+\mu\left(T_{-1}^{-1}(A)\right)+2\mu\left(T_0^{-1}(A)\right).
\]
for $A\subset I$. Note that we only define $\Phi$ on measures supported on $I$ and define $\Phi(\mu)$ to also be supported on $I$, we do not spread mass outside of $I$.

If we set $\mu_0=\delta_{\underline 0}$ then 
\[
\mu_{n}=\Phi(\mu_{n-1})
\]
for $n\in\mathbb N$. Let $|\mu|:=\mu(I)$ denote the total mass of a measure $\mu$ supported on $I$.  Phrased in this new language, Theorem \ref{CountingTheorem} yields the following corollary. 
\begin{cor}\label{Bound implies abs c}
Suppose that there exists a constant $C>0$ such that
\[
|\Phi^n(\delta_{\underline 0})|\leq C\left(\frac{4}{|\gb_1\cdot...\cdot\gb_{d}|}\right)^n
\]
for all $n\in\mathbb N$. Then the self-affine measure $\nu_{\underline{\beta}}$ is absolutely continuous.
\end{cor}

We now turn to understanding how measures grow under the operator $\Phi$.
\begin{lemma}\label{MassGrowth}
\[|\Phi(\mu)|=\mu(T_1^{-1}(I))+\mu(T_{-1}^{-1}(I))+2\mu(T_0^{-1}(I)).\]
\end{lemma}
\begin{proof}
This is immediate from the definition of $\Phi$.
\end{proof}
Define a step function $g:I\to\mathbb R$ by 
\[
g(x) = \chi_I(T_1(x))+\chi_I(T_{-1}(x))+2\chi_I(T_0(x))
\]
Then the previous lemma just says that
\[
|\phi(\mu)|=\int gd\mu.
\]

We have the following theorem.
\begin{theorem}\label{weakconvwithrate}
Suppose that there exists a constant $C>1$ such that
\[
\sum_{n=1}^{\infty}\log \left(\frac{|\gb_1\cdot...\cdot\gb_{d}|}{4}\frac{1}{|\mu_n|}\int gd\mu_n\right)\leq\log(C).
\]
Then the self-affine measure $\nu_{\underline{\beta}}$ is absolutely continuous.
\end{theorem}

Note that $\frac{1}{|\mu_n|}\int gd\mu_n$ is the integral of $g$ with respect to the probability measure $\frac{1}{|\mu_n|}{\mu_n}$. Secondly, if $\mathcal L$ denotes Lebesgue measure on $I$, normalised to have mass one, then $\int_I g(x)d\mathcal L(x)=\frac{4}{|\gb_1\cdot...\cdot\gb_{d}|}$. Thus, if the sequence of probability measures $\frac{\mu_n}{|\mu_n|}$ converge weakly to normalised Lebesgue measure $\mathcal L$ then 
\[
\log \left(\frac{|\gb_1\cdot...\cdot\gb_{d}|}{4}\frac{1}{|\mu_n|}\int gd\mu_n\right)\to 0.
\]
Thus the condition in Theorem \ref{weakconvwithrate} would follow from the sequence $\frac{\mu_n}{|\mu_n|}$ converging weakly to $\mathcal L$ with a given rate.

\begin{proof}
From Corollary \ref{Bound implies abs c} it is enough to prove that 
\begin{align*}
    \frac{1}{n}\log(|\mu_n|)\leq\frac{C}{n}+\log\left(\frac{4}{|\gb_1\cdot...\cdot\gb_{d}|}\right)
\end{align*}

for some $C>0$. From Lemma \ref{MassGrowth} and the discussion afterwards, for each positive integer $k$,
\begin{align*}
    \frac{|\mu_{k+1}|}{|\mu_{k}|}=\frac{|\Phi(\mu_k)|}{|\mu_k|}=\frac{1}{|\mu_k|}\int gd\mu_k.
\end{align*}

Then since $\log(|\mu_0|)=0$, we have
\begin{eqnarray*}
\log(|\mu_n|)&=&\sum_{k=0}^{n-1}\log\left(\frac{|\mu_{k+1}|}{|\mu_{k}|}\right)\\
&=& \sum_{k=0}^{n-1}\log\left(\frac{1}{|\mu_k|}\int gd\mu_k\right)\\
&=& \sum_{k=0}^{n-1}\log\left(\frac{4}{|\gb_1\cdot...\cdot\gb_{d}|}\right)+\sum_{k=0}^{n-1}\log \left(\frac{|\gb_1\cdot...\cdot\gb_{d}|}{4}\frac{1}{|\mu_k|}\int gd\mu_k\right)\\
&\leq & n\log \left(\frac{4}{|\gb_1\cdot...\cdot\gb_{d}|}\right)+\log(C)
\end{eqnarray*}
by the assumption in the theorem. Then
\[
\frac{1}{n}\log(|\mu_n|)\leq \log \left(\frac{4}{|\gb_1\cdot...\cdot\gb_{d}|}\right)+\frac{\log(C)}{ n}
\]
as required.

\end{proof}


\section{The limit measure $\bar{\mu}$}\label{Sec4}
In this section we link the measures $\mu_n$ with methods appeared in \cite{Measures}. The goal is to replace the measures $\mu_n$, which evolve in time, with a fixed limit measure $\bar{\mu}$ supported on a cut and project set. This allows us in section \ref{Sec5} to relate the absolute continuity of $\nu_{\beta}$ to the ergodic theory of cocycles over domain exchange transformations. 

First we need to move in a higher dimensional space by considering the rest of the Galois conjugates $\beta_{d+1},...,\beta_{d+s+1}$.  We set  $\bar{\beta}^n=(\gb_1^n,...,\gb_{d+s+1}^n)$. Set $\bar{T}_i(x_1,...,x_{d+s+1})=(\beta_1x_1+i,...,\beta_{d+s+1}x_{d+s+1}+i)$ which acts on the space $\bar{\mathbb{K}}:=\prod_{i=1}^{d+s+1}\mathbb{F}_{\gb_i}$. We also define the set 

\begin{align*}
\bar{Z}=\{a_{d+s}\bar{\gb}^{d+s}+...+a_0\bar{\gb}^0:a_{d+s},...,a_0\in\mathbb{Z}\}.
\end{align*}

 The set $\bar{Z}$ is a lattice in $\bar{\mathbb{K}}\cong\mathbb{R}^{\sum_{i=1}^{d+s+1}\text{dim}(\mathbb{F}_{\gb_i})}$. That is because $\{\bar{\beta}^0,...,\bar{\beta}^{d+s}\}$ is an independent subset of the real vector space  $\bar{\mathbb{K}}$. That can be checked using the formula for the determinant of the Vandermonde matrix. We partition our coordinates into expanding directions $1,\cdots,d$, contracting directions $d+1,\cdots,d+s$ and the free direction $d+s+1$. The dynamics we will introduce is also expanding on the free direction, but we deal with this coordinate separately since we will eventually project in this direction. 

We define projections $\pi_{e}, \pi_c$ and $\pi_{free}$ from $\bar{\mathbb K}$ onto subspaces of $\bar{\mathbb K}$ corresponding to expanding directions, contracting directions and the free direction respectively. They are given by
\begin{align*}
\pi_e(x_1,\cdots,x_{d+s+1})&=(x_1,\cdots,x_d)\\
\pi_c(x_1,\cdots,x_{d+s+1})&=(x_{d+1},\cdots,x_{d+s})\\
\pi_{free}(x_1,\cdots,x_{d+s+1})&=x_{d+s+1}.
\end{align*}

It is worth noting that $\pi_e,\pi_c$ and $\pi_{\free}$ are injective when restricted to $\bar{Z}$. We define a strip $S\subset \bar{\mathbb K}$ by
\[
S=\{(x_1,\cdots,x_{d+s+1})\in\bar{\mathbb K}:\pi_e(x_1,\cdots,x_{d+s+1})\in I\}.
\]
The following definitions differ from those in \cite{Measures} in that we restrict both $\bar{\mu}_n$ and $\bar{X}$ to the set $S$. Let the measure $\bar{\mu}_n$ on S be given by
 \begin{align*}
 \bar{\mu}_{n}(x)=&\#\left\{(a_1,...,a_{n},b_1,...,b_{n})\in\{0,1\}^{2n}:\sum_{i=1}^{n}(a_i-b_i)\bar{\gb}^{n-i}=x\right\}
 \end{align*}
for $x\subset S$. We do not give mass to points outside $S$. The measure $\bar{\mu}_n$ is a weighted sum of Dirac masses supported on the set
 \begin{align*}
 \bar{X}&:=\left\{\sum_{i=1}^{n}a_i\bar{\gb}^{n-i}:n\in\mathbb{N},a_1...,a_n\in\{-1,0,1\}\right\}\cap S
 \\&=\left\{\bar{T}_{a_n}\circ...\circ \bar{T}_{a_1}(0):n\in\mathbb{N},a_1...,a_n\in\{-1,0,1\}\right\}\cap S,
 \end{align*}

 Notice that for each $i\in\mathbb{Z}$ we have $\bar{T}_i(\bar{Z})\subseteq\bar{Z}$. In particular $\bar{X}\subseteq\bar{Z}$ so $\bar{X}$ 
 is uniformly discrete in $\bar{\mathbb K}$. Note that for $A\subset \bar{\mathbb K}$, $\mu_n\circ \pi_e(A)=\bar{\mu}_n(A)$ so the measures $\bar{\mu}_n$ are just lifts of the measures $\mu_n$ of the previous section to a higher dimensional space in which they are uniformly discrete.

 \begin{definition}
  Let $\mathcal R\subseteq I_{\beta_{d+1}}\times...\times I_{\beta_{d+s}}$ be the attractor of the iterated function system involving the maps $\overline T_i$ restricted to contracting coordinates $d+1,\cdots,d+s$. 
 \end{definition}
 
 The significance of the set $\mathcal{R}$ becomes clear in the condition \label{condition 1} below, although one can already observe that 
 
 \[\bar{X}\subseteq\{z\in\bar{Z}:  \pi_c(z)\in \mathcal R, \pi_e(z)\in I\}.
\]

We will need the following condition which can be checked in finite time (see \cite{Measures}) and which holds for all examples we have checked. 
 
 \begin{condition}\label{condition 1}
	$\bar{X}=\bar{Z}\cap\pi_c^{-1}(\operatorname{int}(\mathcal{R}))\cap S$,
\end{condition}

Below we have plotted on approximation of $\mathcal{R}$ for the example of section \ref{example}.

\begin{figure}[htbp]
\centerline{\includegraphics[scale=.3]{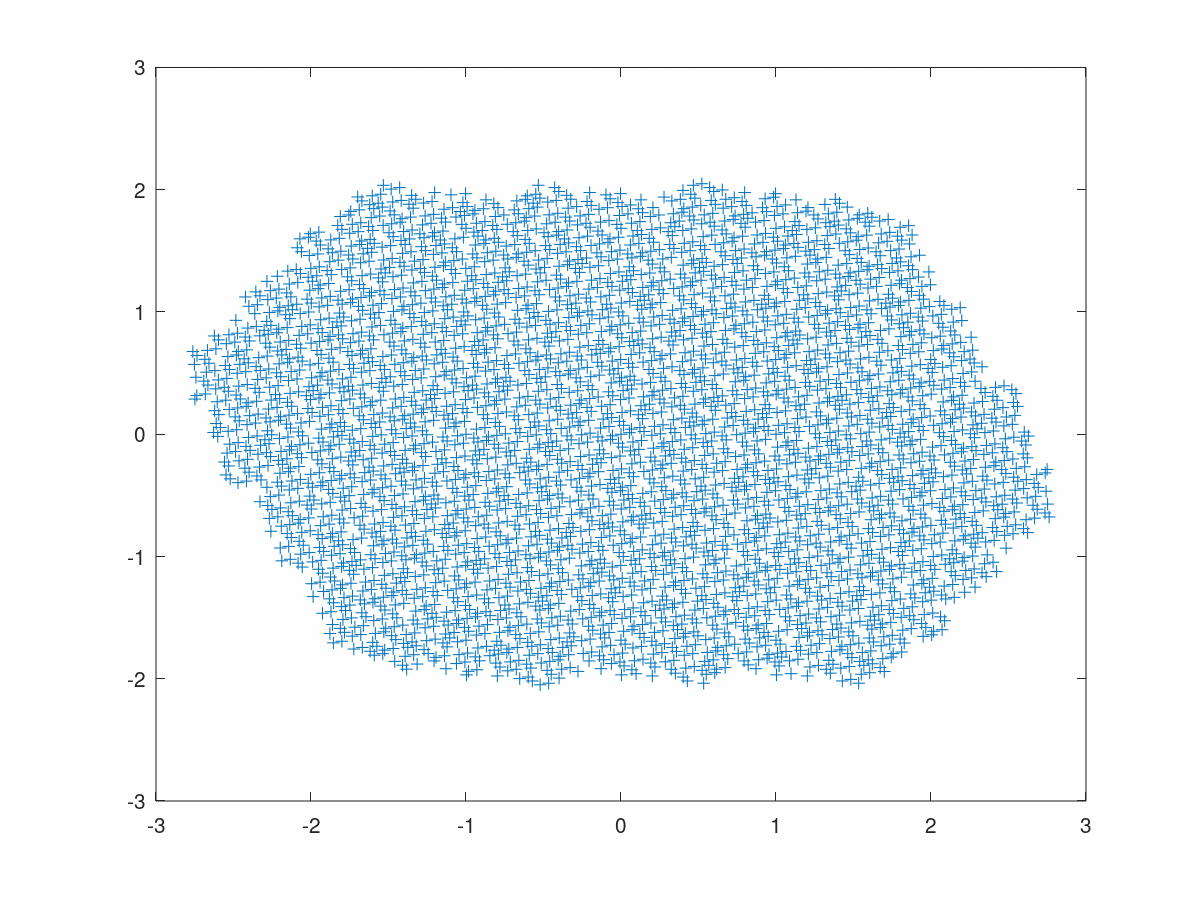}}
\caption{An approximation of $\mathcal{R}$ when   $\beta^4=\beta^3+\beta^2-\beta+1$.}
\label{fig}
\end{figure}

The following theorem recalls some results of \cite{Measures} that we will need.
\begin{theorem}\label{CollectedThm} \text{ }
\begin{enumerate}
\item There exists $\lambda>1$ and a function $f:\bar{X}\to(0,\infty)$ such that for each $x\in \bar{X}$ the sequence of real numbers $\frac{1}{\lambda^n}\bar{\mu}_n(x)$ converges to $f(x)$.
\item We have $0<f(x)\leq f(0)$ for each $x\in\bar{X}$.

\end{enumerate}
\end{theorem}

\begin{definition}
Define the measure $\bar{\mu}$ on $\bar{X}$ by $\bar{\mu}(A)=\sum_{x\in \bar{X}\cap A}f(x)$.
\end{definition}

As we did with the measures $\mu_n$ we define an operator $\bar{\Phi}$ acting on measures on $\bar{X}$ by 
  \begin{align*}
  \bar{\Phi}(\mu)(A)=\mu(\bar{T}_{-1}^{-1}(A))+2\mu(\bar{T}_{0}^{-1}(A))+\mu(\bar{T}_{-1}^{1}(A))
  \end{align*}
  for $A\subset S$, and $\bar{\Phi}(\mu)(A):=\bar{\Phi}(\mu)(A\cap S)$ for more general $A$. $\Phi$ does not spread mass outside of the strip $S$. We have 
  \begin{align*}
      \bar{\mu}_n=\bar{\Phi}^n\delta_0
  \end{align*}
  and
  \begin{align*}
  \bar{\mu}=\frac{1}{\lambda}\bar{\Phi}(\bar{\mu}),
  \end{align*}
  
see Lemma 4.3 of \cite{Measures}. 

We comment that the set $\bar{X}$ is bounded in the coordinates $1,\cdots, d$ since we insist on remaining in the strip $S$, and it is bounded in the coordinated $d+1,\cdots,d+s$ since the action of the maps $\bar{T}_i$ is contracting on these coordinates and orbits remain in the fractal $\mathcal R$. It is only the free direction $d+s+1$ in which $\bar{X}$ is unbounded. 

Let
\[
R_n=\{x\in\bar{X}: |\pi_{\free}(x)|\leq \sum_{i=0}^{n-1}|\beta_{d+s+1}^i| \}
\]

The rest of this section is dedicated to proving the following theorem, which replaces the $\mu_n$ of Theorem \ref{weakconvwithrate} with $\pi_e(\bar{\mu}|_{R_n})$.

\begin{theorem}\label{moving to w}
Suppose that $\lambda<4/|\beta_1\cdots \beta_d|$ and that there exists a constant $C$ such that

\[
\sum_{n=1}^{\infty}\log \left(\frac{|\gb_1\cdot...\cdot\gb_{d}|}{4}\frac{1}{|\bar{\mu}|_{R_n}|}\int gd \pi_e\bar{\mu}|_{R_n}\right)\leq\log(C).
\]

Then the self-affine measure $\nu_{\underline{\beta}}$ is absolutely continuous.
\end{theorem}

Again, we comment that this is really an equidistribution result, requiring that for the probability measure $\frac{1}{|\bar{\mu}|_{R_n}|}\pi_e(\bar{\mu}|_{R_n})$ the mass of certain intervals (involved in the definition of the step function g) is sufficiently close to the Lebesgue measure of those intervals.

\subsection{Proof of Theorem \ref{moving to w}}

In Theorem \ref{CountingTheorem} we gave a criteria for the absolute continuity of $\nu_{\underline{\beta}}$ in terms of the measure $\mu_n$, which can be easily translated to a criteria involving $\bar{\mu}_n$. In order to relate this to $\bar{\mu}$, we need first to consider the subset of $\bar{X}$ upon which $\bar{\mu}_n$ is supported. 

Note that in the free direction our maps $\bar{T}_i$ act by $x\to \beta_{d+s+1}(x)+i$, and so points $\bar{T}_{a_n}\circ \cdots \circ \bar{T}_{a_1}(0)$ must lie in $R_n$. We have the following lemma.
\begin{lemma}\label{m bounds mn}
\[
|\bar{\Phi}^n(\delta_0)|\leq\frac{\lambda^n}{\bar{\mu}(0)}\bar{\mu}(R_n).\]
\end{lemma}

\begin{proof}
Since $\bar{\Phi}$ is monotone and $\bar{\mu}(0)\delta_0\leq \bar{\mu}$, using $\bar{\Phi}(\bar{\mu})/\lambda=\bar{\mu}$  we have 
\begin{align*}
    \frac{1}{\lambda^n}\bar{\Phi}^n(\bar{\mu}(0)\delta_0)\leq \frac{1}{\lambda^n}\bar{\Phi}^n(\bar{\mu})=\bar{\mu}.
\end{align*}
On the other hand from the construction of $R_n$ we have that 

\[\frac{1}{\lambda^n}\bar{\Phi}^n(\bar{\mu}(0)\delta_0)(\bar{X}\setminus R_n)=0.\]

Combining these facts gives
\begin{align*}
    |\bar{\Phi}^n(\delta_0)|=\frac{\lambda^n}{\bar{\mu}(0)}\frac{1}{\lambda^n}\bar{\Phi}^n(\bar{\mu}(0)\delta_0)(R_n)\leq\frac{\lambda^n}{\bar{\mu}(0)}\bar{\mu}(R_n).
\end{align*}
\end{proof}

\begin{lemma}\label{R_n big}
Assume that $\lambda<\frac{4}{|\beta_1\cdots \beta_d|}$. Then $\bar{\mu}(R_n)$ grows exponentially in $n$.
\end{lemma}

\begin{proof}
We note that the $2^n$ rectangles $(T_{a_1}\circ \cdots T_{a_n})^{-1}(I^+)$ are each contained in $I^+$ and each have an area of $\frac{1}{|\beta_1\cdots \beta_d|^n}\times$ Area$(I^+)$, giving a total area of $\frac{2^n}{|\beta_1\cdots \beta_d|^n}\times$ Area$(I^+)$. A lower bound for the total number of overlaps comes from assuming these rectangles are evenly spread, in which case one would have that a typical rectangle intersects $\frac{2^n}{|\beta_1\cdots \beta_d|^n}$ others, giving $\mathcal N_n\geq \frac{1}{2}\frac{4^n}{|\beta_1\cdots \beta_d|^n}$.

Then 
\[\bar{\mu}(R_n)\geq \frac{\mathcal N_n}{\lambda^n}=\frac{1}{2}\left(\frac{4}{|\beta_1\cdots \beta_d|}\frac{1}{\lambda}\right)^n.\] 
which grows exponentially by our assumption.
\end{proof}
We stress that $\lambda$ can be computed by a finite calculation when $\beta$ has no Galois conjugates of absolute value $1$ (as we are assuming throughout this article). Values of $\lambda$ are computed for many values in $\cite{AFKP}$ and in all examples we have computed satisfy the condition of Lemma \ref{R_n big}.

\begin{lemma}\label{4.4}
There exist $\epsilon_n$ tending to zero exponentially quickly such that
\begin{eqnarray*}\bar{\mu}(R_{n+1})&\leq&\frac{1+\epsilon_n}{\lambda}|\bar{\Phi}(\bar{\mu}|_{R_n})|\\
&=&\frac{(1+\epsilon_n)}{\lambda} \int g d\pi_e(\bar{\mu}|_{R_n})
\end{eqnarray*}

\end{lemma}
\begin{proof}
Let $x\in \bar{X}$ be such that 
\begin{align*}
    |\pi_{free}(x)|\leq -2+\sum_{i=0}^n|\beta_{d+s+1}^i|.
\end{align*}

Then \[|\pi_{free}(\bar{T}_i^{-1}(x))|=\left|\frac{\pi_{free}(x)-i}{\beta_{d+s+1}}\right|\leq \sum_{i=0}^{n-1}|\beta_{d+s+1}^i|\] 

and so $\bar{T}_i^{-1}(x)\in R_n\cup (\bar{\mathbb{K}}\setminus\bar{X})$ for each $i\in\{-1,0,1\}$. Hence from $\frac{\bar{\Phi}(\bar{\mu})}{\lambda}=\bar{\mu}$ we get
\begin{align*}
    \frac{1}{\lambda}\bar{\Phi}(\bar{\mu}|_{R_n})(x)&=\frac{1}{\lambda}\left(\bar{\mu}|_{R_n}(\bar{T}_{-1}^{-1}(x))+2\bar{\mu}|_{R_n}(\bar{T}_{0}^{-1}(x))+\bar{\mu}|_{R_n}(\bar{T}_{1}^{-1}(x))\right)
    \\&=\frac{1}{\lambda}\left(\bar{\mu}(\bar{T}_{-1}^{-1}(x))+2\bar{\mu}(\bar{T}_{0}^{-1}(x))+\bar{\mu}(\bar{T}_{1}^{-1}(x))\right)\\
    &=\frac{1}{\lambda}\bar{\Phi}(\bar{\mu})(x)=\bar{\mu}(x).
\end{align*}
Thus
\begin{align}\label{most points are good}
    \notag\bar{\mu}&\left(\left\{x\in R_{n+1}: |\pi_{free}(x)|\leq -2+\sum_{i=0}^n|\beta_{d+s+1}^i|\right\}\right)
    \\=&\frac{1}{\lambda}\bar{\Phi}(\bar{\mu}|_{R_n})\left(\left\{x\in R_{n+1}: |\pi_{free}(x)|\leq -2+\sum_{i=0}^n|\beta_{d+s+1}^i|\right\}\right).
\end{align}

The diameter of 
\begin{align*}
    \left\{x\in R_{n+1}: |\pi_{free}(x)|> -2+\sum_{i=0}^n|\beta_{d+s+1}^i|\right\}
\end{align*}

is uniformly bounded so there is $M>0$ that depends only on $\beta$ such that 
\begin{align*}
    \#\left\{x\in R_{n+1}: |\pi_{free}(x)|> -2+\sum_{i=0}^n|\beta_{d+s+1}^i|\right\}<M
\end{align*}

for all $n\in\mathbb{N}$. By Theorem \ref{CollectedThm} we have $\bar{\mu}(x)\leq \bar{\mu}(0)$ for all $x\in \bar{X}$ and so 
\begin{align}\label{bad bit is bounded}
    \bar{\mu}\left(\left\{x\in R_{n+1}: |\pi_{free}(x)|> -2+\sum_{i=0}^n|\beta_{d+s+1}^i|\right\}\right)<M\bar{\mu}(0).
\end{align}

Combining (\ref{most points are good}) and (\ref{bad bit is bounded}) we have
\begin{align*}
    \bar{\mu}(R_{n+1})&\leq \frac{1}{\lambda}\bar{\Phi}(\bar{\mu}|_{R_n})(\bar{X})+M\bar{\mu}(0)\\
    &\leq \frac{1}{\lambda}\bar{\Phi}(\bar{\mu}|_{R_n})(\bar{X})(1+\epsilon_n)
\end{align*}

Where $\epsilon_n=\frac{M\bar{\mu}(0)}{\frac{1}{\lambda}\bar{\Phi}(\bar{\mu}|_{R_n})(S)}$ tends to zero exponentially fast due to Lemma \ref{R_n big}. 

Finally we mention that, by the construction of $\bar{\Phi}$
\[|\bar{\Phi}(\bar{\mu}|_{R_n})|=\int g d\pi_e(\bar{\mu}|_{R_n}),\]
this is is just the analogue of Lemma \ref{MassGrowth} for the lifted operator $\bar{\Phi}$ rather than $\Phi$.
\end{proof}

\begin{prop}\label{equid on w bounds the errors}

If $\lambda<\frac{4}{|\beta_1\cdots \beta_d|}$ there is $c>1$ such that  

\begin{align*} 
    |\Phi^n(\delta_0)|\leq c\frac{\bar{\mu}(R_0)}{\bar{\mu}(0)}\prod_{i=0}^{n-1} \frac{1}{\bar{\mu}(R_i)}\int gd\pi_e(\mu|_{R_i}).
\end{align*}

\end{prop}

\begin{proof}
From  Lemma \ref{4.4} we have 
\[
\lambda \frac{\bar{\mu}(R_{n+1})}{\bar{\mu}(R_n)}\leq (1+\epsilon_n)\frac{\int gd\pi_e(\bar{\mu}|_{R_n})}{\bar{\mu}(R_n)}.
\]





The above combined with Lemma \ref{m bounds mn} leads to

\begin{eqnarray*}
|\Phi^n(\delta_0)|&=& |(\bar{\Phi}^n(\delta_0))|\\
&\leq& \frac{\lambda^n}{\bar{\mu}(0)}\bar{\mu}(R_n)\\
&=& \frac{\bar{\mu}(R_0)}{\bar{\mu}(0)}\prod_{i=0}^{n-1}\frac{\lambda \bar{\mu}(R_{i+1})}{\bar{\mu}(R_i)}\\
&\leq& \frac{\bar{\mu}(R_0)}{\bar{\mu}(0)}\left(\prod_{i=0}^{n-1} (1+\epsilon_i)\frac{1}{|\bar{\mu}(R_i)|}\int gd\pi_e(\bar{\mu}|_{R_i})\right).
\end{eqnarray*}

The proof is complete by observing that from Lemma \ref{R_n big}  we have 
\[
\prod_{i=0}^{\infty}(1+\epsilon_i)<\infty.
\]

\end{proof}

We can now prove Theorem \ref{moving to w}.  Assuming, as in the theorem, that
\[
\sum_{n=1}^{\infty}\log \left(\frac{|\gb_1\cdot...\cdot\gb_{d}|}{4}\frac{1}{\bar{\mu}(R_n)}\int gd \pi_e(\bar{\mu}|_{R_n})\right)\leq\log(C)
\]

gives 

\[
\prod_{i=0}^{n-1} \frac{1}{|\bar{\mu}(R_n)|}\int gd \pi_e(\bar{\mu}|_{R_n})\leq C\left(\frac{4}{|\gb_1\cdot...\cdot\gb_{d}|}\right)^n,
\]

hence, by Proposition \ref{equid on w bounds the errors},
\[
\mathcal N_n=|\phi^n(\delta_0)|\leq C'\left(\frac{4}{|\gb_1\cdot...\cdot\gb_{d}|}\right)^n
\] for some $C'>0$. Thus the conditions of  Corollary \ref{Bound implies abs c} are satisfied and so the measure $\nu_{\underline{\beta}}$ is absolutely continuous. This completes the proof of Theorem \ref{moving to w}.

\section{Domain Exchange Transformation}\label{Sec5}
In this section we introduce domain exchange transformations, allowing us to state and prove Theorem \ref{FinThm}.

\begin{definition}
We define the set  the successor function $\succf:\bar{X}\rightarrow \bar{X}$ by

\begin{align*}
    \pi_{\free}(\succf(x))=\min\{\pi_{\free}(y):y\in \bar{X},\pi_{\free}(y)>\pi_{\free}(x)\}.
\end{align*}

 \end{definition}

We will later see that the successor function projects to a domain exchange transformation on $D=I\times\mathcal R$. We clarify that in our context a domain exchange transformation is defined as follows.

\begin{definition}
Let $E$ be a compact subset of a euclidean space and $T:E\rightarrow E$. The map $T$ is call a domain exchange transformation if there are $E_1,...,E_n$ measurable subsets of $E$ such that following hold.
\begin{itemize}
    \item $\{E_1,...,E_n\}$ is a partition of $E$.
    \item The map $T$ is an injection. 
    \item If $i\in\{1,...n\}$ then $T|_{D_i}$ is a translation.
\end{itemize}
\end{definition}

Let $\pi_D:\bar{X}\to D$ be given by $\pi_D(x_1,\cdots x_{d+s+1})=(x_1, \cdots x_{d+s})$. Again we notice that $\pi_D|_{\bar{Z}}$ is injective. 

\begin{definition}\label{wndef}
Let $w_n$ be the measure on $D$ defined by 
 \begin{align*}
 w_n=\sum_{\kappa=0}^m\bar{\mu}(\succf^\kappa(0))\delta_{\pi_D(\succf^\kappa(0))},
 \end{align*}
 
 where $m$ is the greatest natural number such that
 
 \[\pi_{\free}(\succf^m(0))\leq\sum_{i=0}^{n-1}|\beta_{d+s+1}^i|.\]
\end{definition}
$w_n$ is the image under projection onto coordinates $1,\cdots,d+s$ of the measure $\bar{\mu}$ restricted in the free direction to the range $[0,\sum_{i=0}^{n-1}|\beta_{d+s+1}^i|]$.

Theorem \ref{moving to w} gave sufficient conditions for the absolute continuity of $\nu_{\underline{\beta}}$ in terms of convergence to Lebesgue of the measures $\pi_{e}w_n$, which were projections onto expanding coordinates $1,\cdots, d$ of the measure $\bar{\mu}$ restricted to a bounded region in the free direction.

Here we stress that the successor function projects to a domain exchange transformation on $I\times\mathcal R$. 



Recall that $D=I\times\mathcal R$.

\begin{definition}
Let 

\begin{align*}
    W=\{x\in\bar{\mathbb{K}}:\pi_c(x)\in\operatorname{int}(\mathcal{R}),\pi_e(x)\in I\}
\end{align*}

and define $T':D\rightarrow\bar{Z}$ by $T'(x)=u$ where 

\begin{align*}
    \pi_{\free}(y+u)=\min\left\{\pi_{\free}(z):z\in \left(y+\bar{Z}\right) \cap W\text{ and }\pi_{\free}(z)>\pi_{\free}(y)   \right\}
\end{align*}

for any $\pi_D(y)=x$.
\end{definition}

It follows from the geometry of $W$ that $T'$ is well defined and that $T'(D)$ is finite. So there are $D_1,...,D_N\subseteq D$ and $u_1,...,u_N\in\bar{Z}$ such that $\{D_1,...,D_N\}$ is a partition of $D$ and 

\begin{align*}
    x\in D_i \Rightarrow T'(x)=u_i.
\end{align*}

Notice that when $x\in S\cap\bar{Z}$ then $x+T'(\pi_D(x))=\succf(x)$.

\begin{lemma}\label{T is DET}
The map $T:D\rightarrow D$ defined by 

\begin{align*}
    T(x)=x+\pi_D(T'(x))
\end{align*}

defines a domain exchange transformation $(T,D_1,...,D_N)$.
\end{lemma}

\begin{proof}
We only need to prove that $T$ is injective. Let, aiming for a contradiction, $x,y\in D$ such that $T(x)=T(y)$. We can choose $x',y'\in S$ with $\pi_D(x')=x$ and $\pi_D(y')=y$ such that $x'+T'(x)=y'+T'(y)$ since $\pi_D(x'+T'(x))=T(x)=T(y)=\pi_D(y'+T'(y))$ and we can freely determine $\pi_{free}(x')$ and $\pi_{free}(y')$. Notice that $y'=x'+T'(x)-T'(y)\in x'+\bar{Z}$ so $x'\neq y'\Rightarrow \pi_{\free}(x')\neq\pi_{\free}(y')$. Assume, without loss of generality, that $\pi_{\free}(y')<\pi_{\free}(x')$. We have $\pi_{\free}(y')<\pi_{\free}(x')<\pi_{\free}(x'+T'(x))=\pi_{\free}(y'+T'(y))$ which contradicts the definition of $T'$ since $x'=y'+T'(y)-T'(x)\in y'+\bar{Z}$.
\end{proof}

Notice that, under condition \ref{condition 1}, $\pi_D(\succf^n(0))=T^n(0)$ since Theorem \ref{CollectedThm} implies $\bar{X}=\bar{Z}\cap W$. For $x\in D$, we define $s(x)$ to be the unique $i$ such that $x\in D_i$. Now we move on to give a characterization of the measures $w_n$ which shows that they have a special structure that could be used to prove equidistribution properties, such as theorem \ref{moving to w} demands for the absolute continuity of $\nu_\beta$. The main ingredient of the proof is theorem 1.3 of \cite{Measures}. For this reason we need to impose the same condition which appears in that theorem and define the set $\Delta$ which also appears in it, as we do below.

\begin{definition}\label{DeltaDef} Let
\begin{align*}
\Delta=\{x-y: &x,y\in \bar{X} \text{ and } \\
&\exists c_1\cdots c_n, d_1\cdots d_n\in\{-1,0,1\}^n: \bar{T}_{c_n}\circ \cdots \bar{T}_{c_1}(x)=\bar{T}_{d_n}\cdots \bar{T}_{d_1}(y)\}.
\end{align*}
\end{definition}
That is, $\Delta$ is the set of differences between points $x,y\in\bar{X}$ which can be mapped to the same point in the future by the application of maps $\bar{T}_i$. Before we state proposition \ref{structure of w} we set $S_i$ to be the maps  $\bar{T_i}$ restricted to the contracting coordinates $d+1,...,d+s$.

\begin{prop}\label{structure of w}
 Under condition \ref{condition 1}, there are functions $\bar{f}_1,...,\bar{f}_N:\mathcal{R}\rightarrow \mathbb{R}^+$ such that
 
 \begin{itemize}
 
\item[i)]  There exists a word $w$ and constants $C_1>0$, $C_2\in(0,1)$ such that for any $a_1\cdots a_n\in\{-1,0,1\}^n$ which contains $r$ non-overlapping copies of the word $w$, $\bar{f}_i$ varies by at most $C_1C_2^{r-1}$ on $S_{a_1}\circ \cdots \circ S_{a_n}(\mathcal R)$.

\item[ii)]  If $m$ is the greatest natural number such that
 
 \[\pi_{\free}(\succf^m(0))\leq\sum_{i=0}^{n-1}|\beta_{d+s+1}^i|,\]
 
 then

  \begin{align*}
    w_n=\bar{\mu}(0)\sum_{\kappa=0}^m\left(\prod_{i=0}^{\kappa-1}\exp\left(    \bar{f}_{s(T^i(0))}(\pi_c(T^i(0)))\right)\right)\delta_{T^\kappa(0)}
\end{align*}

\end{itemize}
\end{prop}

\begin{proof}
From Theorem 1.3 in \cite{Measures}, for each $i\in\{1,...,N\}$ there are $\bar{f}_i:\mathcal{R}\rightarrow\mathbb{R}^+$ satisfying i) such that $\bar{f}_i(\pi_c(x))=\log(\bar{\mu}(x+u_i))-\log(\bar{\mu}(x))$ for all $x\in\bar{X}$. We construct $f_i$ by writing  $u_i$ as a sum of members of the set $\Delta$ and summing the respective functions given by the theorem. We have

\begin{align*}
\bar{\mu}(\succf^n(0))&=\bar{\mu}(0)\prod_{i=0}^{n-1}\frac{\bar{\mu}(\succf^{i+1}(0))}{\bar{\mu}(\succf^{i}(0))}
\\&=\bar{\mu}(0)\prod_{i=0}^{n-1}\frac{\bar{\mu}(\succf^{i}(0)+u_{s(\pi_D(\succf^i(0)))})}{\bar{\mu}(\succf^{i}(0))}
\\&=\bar{\mu}(0)\prod_{i=0}^{n-1}\exp\left(    \bar{f}_{s(\pi_D(\succf^i(0)))}(\pi_c(\succf^i(0)))\right)
\\&=\bar{\mu}(0)\prod_{i=0}^{n-1}\exp\left(    \bar{f}_{s(T^i(0))}(\pi_c(T^i(0)))\right)
\end{align*}

so if  $m$ is the greatest natural number such that
 
 \[\pi_{\free}(\succf^m(0))\leq\sum_{i=0}^{n-1}|\beta_{d+s+1}^i|\]
 
 then 

 \begin{align*}
     w_n&=\sum_{\kappa=0}^m\bar{\mu}(\succf^\kappa(0))\delta_{\pi_D\circ\text{ }\succf^\kappa(0)}\\
     &=\bar{\mu}(0)\sum_{\kappa=0}^m\left(\prod_{i=0}^{\kappa-1}\exp\left(    \bar{f}_{s(T^i(0))}(\pi_c(T^i(0)))\right)\right)\delta_{T^\kappa(0)},
 \end{align*}
 
 concluding ii).

\end{proof}

Recall that Theorem \ref{moving to w} gave a condition for the absolute continuity of $\nu_{\underline{\beta}}$ in terms of the measures $\pi_e(\bar{\mu})$. In Definition \ref{wndef} we introduced the measures $w_n$ which were projections of weighted Dirac measures along an orbit of the successor function $succ$, and in Proposition \ref{structure of w} we explain how the weights appear as a cocycle over the dynamical system $T$. Combining these ideas in one theorem gives the following.

\begin{theorem}\label{FinThm}
Assume that $\lambda<4/|\beta_1...\beta_d|$ and condition \ref{condition 1} holds. Then there exists a domain $D=I\times\mathcal R$, a domain exchange transformation $T:D\to D$ and a function $f:D\to\mathbb R^+$ with  $f(x)=\exp(\bar{f}_{s(x)}(\pi_c(x)))$ such that  if the projection onto $I$ of the sequence of measures 
\[
w_n=\sum_{i=1}^nf(0)f(T(0))\cdots f(T^{n-1}(0))\delta_{T^{n-1}(0)}
\]
converge to Lebesgue measure sufficiently quickly, in the sense that

\[
\sum_{n=1}^{\infty}\log \left(\frac{|\gb_1\cdot...\cdot\gb_{d}|}{4}\frac{1}{|w_n|}\int gd \pi_e w_n|_{R_n}\right)\leq\log(C),
\]

then the measure $\nu_{\underline{\beta}}$ is absolutely continuous.
\end{theorem}

\begin{proof}
The theorem follows from theorem \ref{moving to w}, lemma \ref{T is DET} and proposition \ref{structure of w} after observing that 

\begin{align*}
\pi_e\bar{\mu}|_{R_n}(x)=\begin{cases} \pi_e w_n(x)+\pi_e w_n(-x),\quad &x\in I\setminus\{0\}\\
\pi_e w_n(0), &x=0
\end{cases}.
\end{align*}
\end{proof}



\bibliographystyle{abbrv} 
\bibliography{ABS}

\end{document}